\documentclass{amsart}
\usepackage[utf8]{inputenc}
\usepackage{tikz,hyperref}
\usetikzlibrary{arrows}
\usepackage{caption}

\newtheorem{theorem}{Theorem}[section]
\newtheorem{proposition}[theorem]{Proposition}
\newtheorem{lemma}[theorem]{Lemma}
\newtheorem{corollary}[theorem]{Corollary}
\theoremstyle{definition}
\newtheorem{remark}[theorem]{Remark}
\newtheorem{definition}[theorem]{Definition}
\newtheorem{notation}[theorem]{Notation}
\newtheorem{example}[theorem]
{Example}

\newcommand{\PP}{\mathcal P}
\newcommand{\supp}{\operatorname{Supp}}
\newcommand{\used}{\hat\mu}
\newcommand{\TT}{\mathbb T}
\newcommand{\RR}{\mathbb R}

\newcommand{\BB}{\mathcal B}

\newcommand{\CC}{\mathbb C}
\newcommand{\ZZ}{\mathbb Z}

\newcommand{\TTT}{\mathcal E} %this is the letter for the polyhedral decomposition given by a polynomial
\renewcommand{\L}{\mathbb Z} %this is the lattice points of a polyhedron

\newcommand{\permop}{\operatorname{tw-perm}}
\renewcommand{\matrix}{tw-matrix}
\newcommand{\matrices}{tw-matrices}
\newcommand{\perm}{tw-permanent}
\newcommand{\kernal}{tw-kernel}
\newcommand{\sing}{tw-singular}
\newcommand{\nsing}{tw-nonsingular}
\newcommand{\minor}{minor}

\newcommand{\smooth}{c-smooth}

\newcommand{\vx}{\mathbf x}
\newcommand{\vy}{\mathbf y}

\newcommand{\ds}{4}
\newcommand{\dcdot}{1.5pt}
\newcommand{\dcgrid}{1}
\newcommand{\markpoint}[1]{\draw [fill] (#1) circle [radius=1.5pt];}

\newcommand{\cxx}{(\ds,\dcgrid)}
\newcommand{\cxz}{(\ds + \dcgrid,\dcgrid)}
\newcommand{\czz}{(\ds + 2*\dcgrid,\dcgrid)}
\newcommand{\cxy}{(\ds +\dcgrid,0)}
\newcommand{\cyz}{(\ds + 2*\dcgrid, 0)}
\newcommand{\cyy}{(\ds + 2*\dcgrid,-\dcgrid)}

\newcommand{\cxxy}{(\ds, 0)}
\newcommand{\cxyy}{(\ds + \dcgrid, -\dcgrid)}
\newcommand{\cxxyy}{(\ds, -\dcgrid)}

\newcommand{\conicdualcomplexdots}{
\draw [fill] \cxx node [above] {$x^2$} circle [radius =\dcdot];
\draw [fill] \cxz node [above] {$x$} circle [radius = \dcdot];
\draw [fill] \czz node [above] {$0$} circle [radius = \dcdot];
\draw [fill] \cxy node [left] {$xy$} circle [radius = \dcdot];
\draw [fill] \cyz  node [right] {$y$} circle [radius = \dcdot];
\draw [fill] \cyy node [below] {$y^2$} circle [radius = \dcdot];
}

\newcommand{\conicdualcomplex}{
\conicdualcomplexdots
\draw (\ds, \dcgrid) -- (\ds + 2*\dcgrid, \dcgrid) -- (\ds + 2*\dcgrid, -\dcgrid) -- (\ds, \dcgrid);
}

\newcommand{\ppdualcomplex}{
\conicdualcomplexdots
\draw [fill] \cxxy node [left] {$x^2y$} circle [radius =\dcdot];
\draw [fill] \cxyy node [below] {$xy^2$} circle [radius = \dcdot];
\draw [fill] \cxxyy node [below] {$x^2y^2$} circle [radius = \dcdot];
}

\newcommand{\standardconic}{
\draw [<->] (1,2) -- (1,1) -- (0,0) -- (0,-1) -- (-1,-2);
\draw [->] (0,-1) -- (2,-1);
\draw [->] (1,1) -- (2,1);
\draw [->] (0,0) -- (-1,0) -- (-1,2);
\draw [->] (-1,0) -- (-3,-2);
\node at (1.5,1.5) {$0$};
\node at (1,0) {$y$};
\node at (.7,-1.5) {$y^{2}$};
\node at (-.7,-.7) {$x y$};
\node at (-1.7,0) {$x^{2}$};
\node at (0, 1) {$x$};
\conicdualcomplex
\draw (\ds + \dcgrid, \dcgrid) -- (\ds + \dcgrid, 0) -- (\ds + 2*\dcgrid, 0) -- (\ds + \dcgrid, \dcgrid);
}

\newcommand{\otherconic}{
\draw [<->] (-2,-.5) -- (-1,.5) --(0,.5) -- (.5,0) -- (.5,-1) -- (-.5,-2);
\draw [->] (-1,.5) -- (-1,1.5);
\draw [->] (0,.5) -- (0,1.5);
\draw [->] (.5,0) -- (1.5,0);
\draw [->] (.5,-1) -- (1.5,-1);
\node at (1,1) {$0$};
\node at (-.5,1) {$x$};
\node at (1,-.5) {$y$};
\node at (-.5,-.5) {$xy$};
\node at (-1.5,1) {$x^2$};
\node at (1,-1.5) {$y^2$};
\conicdualcomplex
\draw (\ds +\dcgrid,\dcgrid) -- (\ds +\dcgrid,0) -- (\ds+2*\dcgrid,0);
\draw (\ds + \dcgrid,0) -- (\ds +2*\dcgrid,\dcgrid);
}

\newcommand{\cxm}[1]{\node [above right] at (\ds + \dcgrid/2, \dcgrid/2) {$#1$};}
\newcommand{\cym}[1]{\node [above right] at (\ds + 3*\dcgrid/2,-\dcgrid/2) {$#1$};}
\newcommand{\coxm}[1]{\node [above left] at (\ds + 3*\dcgrid/2, \dcgrid/2) {$#1$};}
\newcommand{\coym}[1]{\node [below right] at (\ds + 3*\dcgrid/2,\dcgrid/2) {$#1$};}

\title{Determining Tropical Hypersurfaces}
\author{Drew Johnson}

\begin{document}

\begin{abstract}
We consider the question of when points in tropical affine space uniquely determine a tropical hypersurface. We introduce a notion of multiplicity of points so that this question may be meaningful even if some of the points coincide. We give a geometric/combinatorial way and a tropical linear-algebraic way to approach this question.  First, given a fixed hypersurface, we show how one can determine whether points on the hypersurface determine it by looking at a corresponding marking of the dual complex. With a regularity condition on the dual complex and when the number of points is minimal, we show that our condition is equivalent to the connectedness of an appropriate sub-complex. Second, we introduce notions of non-singularity of tropical matrices and solutions to tropical linear equations that take into account our notion of multiplicity and prove a Cramer's Rule type theorem relating them.
\end{abstract}
\maketitle
\section{Introduction}
%\subsection{Classical}
We first recall the classical situation. Given a lattice polytope $\Delta \subset \RR^{n}$, let $\L(\Delta) = \Delta \cap \ZZ^n$ and $|\Delta| = |\ZZ(\Delta)|$. One can consider the linear system of hypersurfaces in $(\CC^{*})^n$ given by Laurent polynomials of the form
\[
    \sum_{I=(i_1, \dots, i_n) \in \L(\Delta)} a_I x_1^{i_1} \cdots x_n^{i_n}.
\]
One can also view these hypersurfaces as lying in the toric variety corresponding to $\Delta$. 

 Requiring the hypersurface to pass through a given point in $(\CC^*)^n$ imposes a linear condition on the coefficients $a_I$, so we see that we require $|\Delta| - 1$ points in general position to uniquely determine the hypersurface (up to uniform scaling of the coefficients). Given an explicit set of $|\Delta|-1$ points, one can check whether these points are general by verifying that the $(|\Delta|-1) \times |\Delta|$ matrix formed by evaluating monomials at the points is full rank. To compute the coefficients, one solves the associated homogeneous linear system. For example, one could accomplish both tasks via Cramer's rule by computing the maximal minors of the matrix.
 
\subsection{Tropical hypersurfaces and higher codimension conditions} \label{mult}
Recall that a tropical polynomial
\[
 f = \bigoplus_{I=(i_1, \dots, i_n) \in \L(\Delta)} a_I \odot x_1^{\odot i_1} \odot \cdots \odot x_n^{\odot i_n}
\]
defines a subset $V(f) \subset \TT^n$ called a tropical hypersurface  by the condition $\vx \in V(f)$ if 
\[
    \min_{I} \left\{a_I + \sum_{k=1}^n i_k x_k\right\}
\]
is achieved by at least two choices of $I$. (The tropical preliminaries needed in this paper will be reviewed in more detail in Section \ref{trophyp}.)

In this paper, we introduce an extension of this notion by saying that $V(f)$ has multiplicity $m$ at $p$ if the minimum is achieved by precisely $m+1$ choices of $I$. Requiring $V(f)$ to have multiplicity at least $m$ at $p$ is a codimension $m$ condition on the coefficients $a_I$. We wish to study when points with assigned multiplicity uniquely determine a hypersurface.
 
\subsection{Geometric/combinatoric}
In our first approach to this problem, we start with a hypersurface with some fixed points and ask whether the hypersurface can be deformed while still containing the points.

To illustrate, consider the curves in Figures \ref{disc} and \ref{goodconic}. In each figure, the curve with its fixed points is shown on the left. The dual complex is shown on the right. Each edge in the dual complex whose corresponding edge in the curve has a fixed point is darkened. One can see from the figures that if the darkened subcomplex is disconnected as in Figure \ref{disc}, then simultaneously decreasing the coefficients of a component will give a deformation. On the other hand, if the darkened subcomplex is connected as in Figure \ref{goodconic}, then the curve is uniquely determined. Summarizing, we have
\begin{proposition} \label{prop:fe}
Let $X$ be a tropical curve with specified fixed points in the interior of its edges. Assume every 2-polytope of the dual complex contains exactly 3 lattice points. Then $X$ is uniquely determined if and only if the corresponding subgraph of the dual complex is connected.
\end{proposition}

Notice that this immediately implies that at least $|\Delta|-1$ points are required for such a hypersurface to be uniquely determined, agreeing well with the classical case.

Proposition \ref{prop:fe} overlaps with results in, e.g., Mikhalkin's seminal paper \cite{mikhalkin_enumerative_2005}, or Gathmann and Markwig's paper \cite{gathmann_numbers_2007}. In papers like these, the interest is in counting the number of curves of a fixed degree and genus through general points and showing that the number is independent of the general points. Our paper rather remains in the situation where only one curve is expected, generalizing Proposition \ref{prop:fe} in different ways. %. We generalize the problem in different ways,by introducing the multiplicities of Section \ref{mult} and investigating hypersurfaces of any dimension.
Our Proposition \ref{hypermult} will give a condition that works for a hypersurface of any dimension and any configuration of points (allowing them to lie in higher codimensional polyhedra) and takes into account the multiplicities of Section \ref{mult}. This condition lacks the appeal of the connectedness condition of Proposition \ref{prop:fe}, but by imposing additional hypotheses, our Theorem \ref{thm:con} recovers an equally attractive (but much more general) statement.

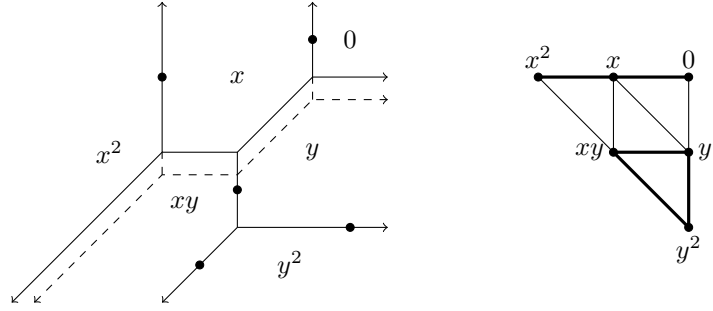
\begin{figure}
\begin{center}
\begin{tikzpicture}
\standardconic
\markpoint{-1,1}
\markpoint{0,-.5}
\markpoint{-.5, -1.5}
\markpoint{1.5,-1}
\markpoint{1,1.5}
\draw [<->, dashed] (-2.7,-2) -- (-1,-.3) -- (0,-.3) -- (1,.7) -- (2,.7);
\draw [dashed] (-1,-.3) -- (-1,0);
\draw [dashed] (1,.7) -- (1,1);
\draw [very thick] (\ds, \dcgrid) -- (\ds + 2*\dcgrid, \dcgrid);
\draw [very thick] (\ds + \dcgrid,0) -- (\ds + 2*\dcgrid,-\dcgrid) -- (\ds + 2*\dcgrid,0) -- (\ds + \dcgrid,0);
\end{tikzpicture}
\caption[A disconnected marked subcomplex gives a deformation]{A disconnected marked subcomplex gives a deformation.}
\label{disc}
\end{center}
\end{figure}

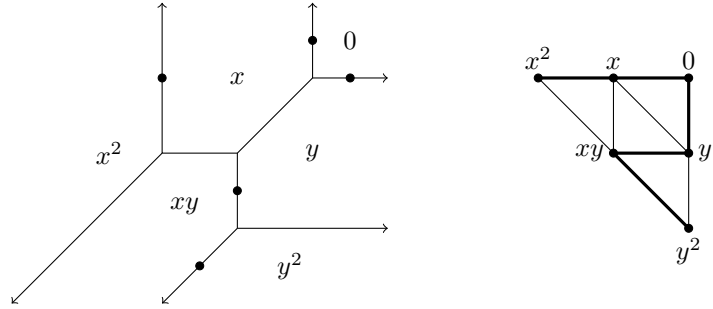
\begin{figure}
\begin{center}
\begin{tikzpicture}
\standardconic
\markpoint{-1,1}
\markpoint{0,-.5}
\markpoint{-.5, -1.5}
\markpoint{1.5,1}
\markpoint{1,1.5}
\draw [very thick] (\ds, \dcgrid) -- (\ds + 2*\dcgrid, \dcgrid) -- (\ds + 2*\dcgrid,0) -- (\ds + \dcgrid,0) -- (\ds + 2*\dcgrid,-\dcgrid);
\end{tikzpicture}
\end{center}
\caption[A connected marked subcomplex]{A connected marked subcomplex.}
\label{goodconic}
\end{figure}

\subsection{Tropical Cramer's rule}
Tropical linear algebra can be subtle. For example, Akian, Bapat, and Gaubert \cite{akian_max-plus_2006} give five different notions of the rank of a tropical matrix, all inequivalent. A program to address some of these issues is given by the supertropical linear algebra of Izhakian \cite{izhakian_supertropical_2013}. 

The tropical linear algebra most relevant to this paper can be found in Richter-Gebert, Sturmfels, and Theobald \cite{richter-gebert_first_2005}, where the authors give an algebraic way to determine whether points uniquely determine a tropical hypersurface, which we review here. A vector $\vx = (x_1, \dots, x_N) \in \TT^N$ is said to be in the \emph{tropical kernel} of a $M \times N$ tropical matrix $A$ if for each row $i$,
\[
    \bigoplus_j x_j \odot A_{ij} = \min_j \{x_j + A_{ij}\}
\]
is achieved at least twice.

Now let $A$ be an $N \times N$ tropical matrix. The \emph{tropical permanent} of $A$ is
\begin{equation}
    \bigoplus_{\sigma \in S_n}  \bigodot_i A_{i,\sigma(i)}  = \min_{\sigma \in S_n} \left\{ \sum_i A_{i, \sigma(i)} \right\}. \label{tropperm}
\end{equation}
We say that $A$ (or its permanent) is \emph{tropically singular} if the minimum in the tropical permanent is achieved at least twice. $A$ is \emph{tropically nonsingular} otherwise. 

The following are the two fundamental results.
\begin{theorem}[\cite{richter-gebert_first_2005}, Lemma 5.1] \label{str_square}
An $N \times N$ tropical matrix has a vector in its tropical kernel if and only if it is tropically singular.
\end{theorem}

A maximal minor of a $(N-1) \times N$ tropical matrix is the tropical permanent of a $(N-1) \times (N-1)$ submatrix obtained by deleting a column.
\begin{theorem}[\cite{richter-gebert_first_2005}, Theorem 5.3] \label{thm:str_n-1xn}
Let $A$ be an $(N-1) \times N$ tropical matrix. Then the vector of maximal minors of $A$ is in its tropical kernel. This vector is the unique vector in the kernel (up to tropical scaling) if and only if every maximal minor is tropically nonsingular.
\end{theorem}

In Section \ref{sec:trop-la}, we introduce \emph{tropical weighted matrices} which take into account the multiplicities defined in Section \ref{mult}. We will state and prove generalizations  of the two theorems above for tropical weighted matrices (Theorems \ref{square} and \ref{n-1xn}). Our method also gives a new, purely combinatorial proof of Theorem \ref{str_square} (which in \cite{richter-gebert_first_2005} uses a lift to Puiseux series). To prove our generalization of Theorem \ref{thm:str_n-1xn} we borrow techniques and terminology from \cite{richter-gebert_first_2005} and \cite{sturmfels_maximal_1993}, generalizing  and specializing to our situation. The restriction of our proof to the case of Theorem \ref{thm:str_n-1xn} gives a self-contained proof that fills some possibly omitted details in \cite{richter-gebert_first_2005} (see Remark \ref{rmk:gap}).

\subsection{Acknowledgments}
This work grew out of discussions with Aaron Bertram, Tyler Jarvis, Lance Miller, and Dylan Zwick. The author also thanks Bernd Sturmfels for email correspondence. The author was supported by NSF Research Training Grant DMS-1246989 during parts of the work on this paper.

\section{Tropical preliminaries} \label{trophyp}
We first quickly review some basic notions from polyhedral geometry. Recall that a \emph{polyhedron} in $\RR^n$ is the solution to a finite set of linear inequalities and equations. A \emph{face} of a polyhedron is a subset of the polyhedron obtained by changing some of the inequalites into equations. A \emph{vertex} of a polyhedron is a zero-dimensional face and an \emph{edge} is a one-dimensional face.  A \emph{polytope} is a compact polyhedron. A \emph{lattice polytope} is a polytope with all of its vertices at integer points. We call a $d$ dimensional polyhedron a $d$-polyhedron. A \emph{polyhedral complex} in $\RR^n$ is a collection of polyhedra $\PP$ satisfying
\begin{itemize}
\item For every face $\sigma$ of a polyhedron $P \in \PP$, we have $\sigma \in \PP$.
\item For $P,P' \in \PP$, we have that $P \cap P'$ is a face of both.
\end{itemize}
The \emph{support} of a polyhedral complex is the union of the polyhedra. A \emph{polyhedral decomposition} of a polyhedron $P$ is a polyhedral complex whose support is equal to $P$.

In tropical geometry we work over the min-plus semiring $\TT$. This ring as a set is the same as $\RR$, but with new operations $\oplus$, $\odot$ defined by
\begin{align*}
a \odot b &= a+b \\
a \oplus b &= \min(a,b).
\end{align*}
Fix a lattice polytope $\Delta$. A polynomial $f \in \TT[x_1, \dots, x_n]$ is a formal sum 
\[
    f = \bigoplus_{I=(i_1, \dots, i_n) \in \L(\Delta)} a_I \odot x_0^{\odot i_1} \odot \cdots \odot x_n^{\odot i_n}
\]
Such a polynomial defines a polyhedral decomposition $\TTT_f$ of $\TT^{n}$: for any $\BB \subset\L(\Delta)$, we get a polyhedron $P_{\BB}$ (possibly empty) of this decomposition containing all points that make the monomials in $\BB$ minimal, that is 
\[
 P_{\BB} = \left\{\vy \in \TT^{n}: a_J + \sum_{k=1}^n j_k y_k = \min_{I \in \L(\Delta)} \left\{a_I + \sum_{k=1}^n i_k y_k\right\} \mbox{ for all } J \in \BB\right\}.
\]
The tropical hypersurface $V(f)$ is the subcomplex of $\TTT_f$ with polyhedra $P_\BB$ with $|\BB| \ge 2$. Its support is the set of $\vy$ such that 
\begin{equation}
    \min_{I} \left\{a_I + \sum_{k=0}^n i_k y_k\right\} \label{eq:min}
\end{equation}
is achieved by more than one choice of $I$. By a slight abuse, we often call this support the tropical hypersurface.  When we need to mention the polytope $\Delta$, we will call $V(f)$ a $\Delta$-hypersurface. See Figure \ref{tropline} for a simple example. 

One can construct a polyhedral decomposition $\PP_f$ of $\Delta$ called the \emph{dual complex}. For each subset of lattice points $\BB \subset \mathbb Z(\Delta)$, the convex hull of $\BB$ is in $\PP_f$ if $P_{\BB}$ is nonempty.

Figure \ref{goodconic} shows the dual complex to a tropical plane curve. In our figures of the dual complex, we prefer to  invert the axes. This way, one can often superimpose the tropical curve on the dual complex so that every vertex lies in its corresponding 2-polyhedron, and every edge is orthogonal to its corresponding edge.

We remark that there is another (equivalent) description of the dual complex using the projection of the lower faces of the convex hull of the points $(i_1, \dots, i_n, a_I) \in \RR^{n+1}$. See \cite{richter-gebert_first_2005} for details.

In this paper, when we give a tropical hypersurface as $V(f)$, we will assume that $f$ is \emph{saturated}, that is, every monomial of $f$ achieves the minimum in \eqref{eq:min} at some point. Geometrically, this is no restriction, since for a nonsaturated polynomial, one can form its \emph{saturation} by decreasing the offending monomials just until the new polynomial is saturated. The saturation of a polynomial still gives the same hypersurface as the original.

\begin{figure} 
\centering
\begin{tikzpicture}
\draw [<->] (0,1) -- (0,0) node [above right] {$(-3,2)$} -- (1,0);
\draw [<-] (-1,-1) -- (0,0);
\node at (.7,.7) {$0$};
\node at (-.5, 0) {$x$};
\node at (0,-.5) {$y$};
\draw [fill] (0,0) circle [radius =1pt];
\end{tikzpicture}
\caption{The tropical line $3\odot x \oplus -2 \odot y \oplus 0$}
\label{tropline}
\end{figure}
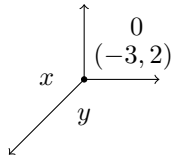

 \begin{definition}
 We say a polyhedral complex is \emph{lattice simplicial} if for every $d$, the $d$-polyhedra of the complex contain exactly $d+1$ lattice points. For this paper, we will call a hypersurface \emph{\smooth{}} if its dual complex is lattice simplicial. We remark that lattice simplicial implies simplicial, but in dimension greater than 2, it is weaker than unimodular. %We will not need those notions in this paper.%Hence \smooth{}ness is weaker than the notion of smoothness.
 \end{definition}

\section{Points on a tropical hypersurface} \label{pointson}
We wish to generalize Proposition \ref{prop:fe} by allowing the fixed points to lie in higher codimension polyhedra and allowing the higher multiplicity incidence conditions.  We also want to be able to consider non-\smooth{} hypersurfaces. Instead of looking at a subgraph, we now consider a weighted subcomplex.
\begin{definition} \label{def:mu}
Let $\Delta$ be a lattice polytope, with $\PP$ a lattice polyhedral decomposition. 

Let $|P|$ be the number of lattice points contained in a lattice polytope $P$.

A \emph{weighting} $\mu$ of $\PP$ is a map
\[
	\mu: \PP \rightarrow \mathbb Z
\]
satisfying
\[
	0 \le \mu(P) \le |P|-1
\]
for all polytopes $P \in \PP$. 

Define
\[
	|\mu| = \sum_{P \in \PP} \mu(P).
\]

Let $\PP^0$ be the vertices of $\PP$. In the lattice simplicial case, of course, all the lattice points $\L(\Delta)$ are vertices of $\PP$, but in general this is not the case. 
\end{definition}

Now we need to define the analog of the connected component of the subgraph in Figure \ref{disc} that provided a deformation.

\begin{definition} \label{def:rigid}
Let $L \subset \L(\Delta)$ be a nonempty subset of the lattice points of $\Delta$ such that $L$ does not contain $\PP^0$. We say that $L$ is \emph{deformable} (with respect to $\mu$) if for all $P \in \PP$ we have either
\[
    |P \cap L| = 0
\]
or
\[
    |P \cap L| \ge \mu(P)+1.
\]
We say $\mu$ is \emph{rigid} if no such $L$ is deformable.
\end{definition}

\begin{proposition} \label{hypermult}
Suppose $V(f)$ is a $\Delta$-hypersurface with dual complex $\PP$. Let $\vx_1, \dots, \vx_K$ be points on $V(f)$ and $m_1, \dots, m_K$ be positive integers such that $V(f)$ has multiplicity at least $m_i$ at $\vx_i$. Assume that no two points lie in the interior of the same polyhedron, and let $\mu$ be the weighting of $\PP$ defined by setting $\mu(P) = m_i$ if $\vx_i$ is in the interior of $P$, and $\mu(P) = 0$ otherwise.

Then the following are equivalent:
\begin{itemize}
	\item $V(f)$ is the unique $\Delta$-hypersurface with multiplicity at least $m_i$ at $\vx_i$ for each $i$.
	\item $\mu$ is rigid.
\end{itemize}
\end{proposition}
\begin{proof}
First we claim that decreasing the coefficients corresponding to a subset $L \subset \ZZ(\Delta)$ by a small amount produces a tropical hypersurface that still satisfies the conditions imposed by the $\vx_i$ and $m_i$ if and only if $L$ is deformable. Indeed, any point $\vx_k$ corresponds to some $P \in \PP$ whose lattice points correspond to the monomials minimized at $\vx_k$. We see then that $\vx_k$ remains on the hypersurface with multiplicity $m_k$ if and only if there are either 0 or at least $m_k+1=\mu(P)+1$ monomials of $P$  being decreased, so the claim is proved. 

If $\mu$ is not rigid, there is a deformable $L$ not containing $\PP^0$. Then the deformation corresponding to $L$ is an actual deformation. (Notice that $\PP^0$ corresponds to monomials which are uniquely minimizing on some top-dimensional polyhedron of $\TTT_f$. If all of these coefficients were decreased, the saturation of the new polynomial would be equal to a rescaling of the original $f$.) Hence $V(f)$ is not uniquely determined.

Now suppose $g$ is another polynomial such that $V(g)$ satisfies the conditions imposed by the $\vx_i$ and is distinct from $V(f)$ posed by the points. Then, for any $t \in \TT$, $f \oplus (t \odot g)$ also satisfies the conditions imposed by the points. For $t \gg 0$, $f \oplus (t \odot g) = f$, while for $t \ll 0$, $f \oplus (t \odot g) = t \odot g$. Hence, for some value of $t$, decreasing $t$ gives a deformation of the type above, which by the first claim in this proof gives a deformable $L$. 
\end{proof}

In specific examples, checking that $\mu$ is rigid may not be as easy as checking whether a graph is connected, as in the case of \smooth{} plane curves (see, for example Figure \ref{fig:looks-con}).  However, if the hypersurface is \smooth{} and the number of points is minimal, Theorem \ref{thm:con} will tell us that the situation is almost as nice as for plane curves.  We first introduce some language.

\begin{definition}
We say that $P \in \PP$ is \emph{full} if $\mu(P)$ is as large as allowed, that is $\mu(P) = |P|-1$. We say $P$ is \emph{deficient} if $0 < \mu(P)<|P|-1$.

We say $\mu$ is \emph{full} if for every $P$, either $\mu(P) = |P|-1$ or $\mu(P) = 0$, or equivalently, no $P$ is deficient.

For any $\mu$, define
\[
	\supp(\mu) = \bigcup_{P \mbox{ full}} P
\]
Notice that every vertex of $\PP$ is full and thus in $\supp(\mu)$.
\end{definition}

\begin{theorem}
\label{thm:con}
Suppose that $\PP$ is a lattice simplicial decomposition of $\Delta$ with weighting $\mu$ (see Definition \ref{def:mu}) and $|\mu| = |\Delta|-1$. Then the following are equivalent:
\begin{itemize}
\item $\mu$ is rigid (see Definition \ref{def:rigid}).
\item $\supp(\mu)$ is connected.
\item $\supp(\mu)$ is connected and $\mu$ is full.
\end{itemize}
\end{theorem}
See Figure \ref{fig:p1xp1} for an example of a curve satisfying the conditions of the theorem.

\begin{figure}
\begin{center}
\begin{tikzpicture}
\path [fill=lightgray] \cxx --\cyy --\czz -- \cxx;
\otherconic
\markpoint{-1,.5}
\markpoint{0,.5}
\markpoint{.5,0}
\markpoint{.5,-1}
\markpoint{0,1}
\draw [very thick] \cxz -- \czz;
\draw [dashed, <->] (-2, -.8) -- (-1,.2) -- (0,.2) -- (.2,0) -- (.2,-1) -- (-.8,-2);
\draw [dashed] (-1,.2) -- (-1,.5);
\draw [dashed] (0,.2) -- (0,.5);
\draw [dashed] (.2,0) -- (.5,0);
\draw [dashed] (.2,-1) -- (.5,-1);
\newcommand{\Lrad}{.1}
\draw [thick] \cxx circle [radius = \Lrad];
\draw [thick] \cxz circle [radius = \Lrad];
\draw [thick] \czz circle [radius = \Lrad];
\draw [thick] \cyz circle [radius = \Lrad];
\draw [thick] \cyy circle [radius = \Lrad];

\cxm{1}
\cym{1}
\coym{1}
\coxm{1}
\end{tikzpicture}
\end{center}
\caption[The marked dual graph looks connected, but is not full or rigid]{The marked dual graph looks connected, but is not full or rigid. The circled vertices give an $L$ that is deformable.}
\label{fig:looks-con}
\end{figure}
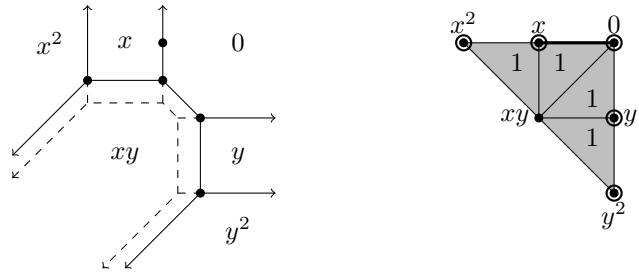

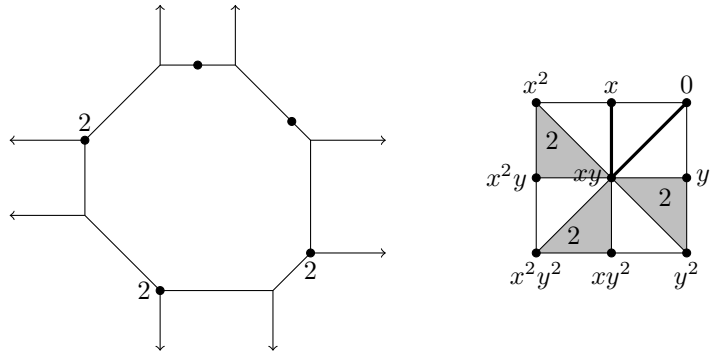
\begin{figure}
\begin{center}
\begin{tikzpicture}
\path [fill=lightgray] \cxx -- \cxxy -- \cxy -- \cxx;
\path [fill=lightgray] \cxxyy -- \cxyy -- \cxy -- \cxxyy;
\path [fill=lightgray] \cxy -- \cyz -- \cyy -- \cxy;
\draw [very thick] \cxz -- \cxy -- \czz;
\ppdualcomplex
\draw \cxx -- \czz -- \cyy -- \cxxyy -- \cxx;
\draw \cxx -- \cyy;
\draw \cxz -- \cxyy;
\draw \cxxy -- \cyz;
\draw \cxxyy -- \czz;
\node [right] at (\ds, \dcgrid/2) {$2$};
\cym{2}
\node [above] at (\ds + \dcgrid/2, -\dcgrid) {$2$};
\draw (-1,1.5) -- (0,1.5) -- (1,.5) -- (1,-1) -- (.5,-1.5) -- (-1,-1.5) -- (-2,-.5) -- (-2,.5) -- (-1,1.5);
\draw [<-] (-3,-.5) -- (-2,-.5);
\draw [<-] (-3, .5) -- (-2, .5);
\draw [<-] (-1,2.3) -- (-1,1.5);
\draw [<-] (0,2.3) -- (0,1.5);
\draw [<-] (-1,-2.3) -- (-1,-1.5);
\draw [<-] (.5, -2.3) -- (.5, -1.5);
\draw [<-] (2,-1) -- (1,-1);
\draw [<-] (2,.5) -- (1,.5);
\node [above] at (-2,.5) {$2$};
\node [left] at (-1,-1.5) {$2$};
\node [below] at (1,-1) {$2$};
\markpoint{-2,.5}
\markpoint{-1,-1.5}
\markpoint{1,-1}
\markpoint{.75,.75}
\markpoint{-.5, 1.5}
\end{tikzpicture}
\end{center}
\caption[A ``smooth (2,2) curve in $\mathbb P^1 \times \mathbb P^1$" with a rigid weighting with multiplicity]{A ``smooth (2,2) curve in $\mathbb P^1 \times \mathbb P^1$" with a rigid weighting with multiplicity.}
\label{fig:p1xp1}
\end{figure}

For the proof, we will need the following.
\begin{definition}
For $L \subset \ZZ(\Delta)$, define
\[
	\used(L) = \sum_{P: |P \cap L | \ge \mu(P)+1} \mu(P).
\]
\end{definition}
Clearly $\used(L) \le |\mu|$. We think of $\used(L)$ as measuring how much of $|\mu|$ has been already ``taken care of" by $L$. The motivation for this definition is the following lemma.

\begin{lemma}
\label{lem:used}
Assume $\PP$ is lattice simplicial and $|\mu|=|\Delta|-1$. Suppose there exists $L$ such that $\used(L) \ge |L|$. Then $\mu$ is not rigid. 
\end{lemma}
\begin{proof}
If $L_0 := L$ is deformable, we are done. Otherwise, there is some polytope $P_0$ with $1 \le |P_0 \cap L_0| \le \mu(P_0)$. Add some vertices to $L_0$ to form $L_1$ so that $|P_0 \cap L_1| > \mu(P_0)$. Notice that $\used(L_1)-\used(L_0) \ge \mu(P_0)$ and $|L_1|-|L_0| \le \mu(P_0)$. It follows from these two inequalities that $\used(L_1) \ge |L_1|$. Continue this process. At each step we have $\used(L_k) \ge |L_k|$. Since $\used$ is bounded above, this process must terminate, that is, eventually one obtains a deformable $L_k$. As $|L_k| \le \used(L_k) \le |\mu| = |\Delta|-1$, this violates rigidity.
\end{proof}

We next isolate a computation to be used in the proof.
\begin{lemma}
\label{lem:new-span}
For a simplicial complex $\mathcal K$
\[
    \sum_{P \mbox{ maximal}} \dim P \ge |\mathcal K^0| - h^0(\mathcal K).
\]
Here maximal simplexes are those which are not contained in a simplex of higher dimension, $|\mathcal K^0|$ is the number of vertices of $\mathcal K$, and $h^0(\mathcal K)$ is the number of connected components.
\end{lemma}
\begin{proof}
Since both sides are additive under disjoint union, it is sufficient to check this for a connected $\mathcal K$. It is trivially true for a complex consisting of a single point. If one attaches a $d$ cell to a complex with the new cell sharing $k$ points with the original complex, then $|\mathcal K_0|$ increases by $d+1-k$. On the other hand, the left hand side increases by at least $d - (k-1)$, since at worst a $(k-1)$-cell is no longer maximal. 
\end{proof}

\begin{corollary} \label{unioncomp}
Let $\mu$ be a weighting of $\mathcal P$. Suppose $L$ is the set of lattice points of some union $\mathcal K$ of $s$ components of $\supp(\mu)$ and suppose that there is a deficient $P' \in \mathcal P$ with $|P' \cap L| \ge \mu(P')+1$. Then
\[
	\used(L) \ge |L| - s + \mu(P').
\] 
\end{corollary}
\begin{proof}
We have 
\[
\used(L) \ge \mu(P') + \sum_{P \subset \mathcal K, \mbox{ maximal}} \mu(P).
\]
Notice that for $P$ in the sum $\mu(P) = \dim P$ and then apply the Lemma.
\end{proof}

\begin{proof}[Proof of Theorem \ref{thm:con}]
Suppose $\supp(\mu)$ is connected. We will show $\mu$ is rigid. Let $L$ be given. Pick some $v_1 \in L$ and $v_2 \notin L$. There must be a sequence $P_0, \dots, P_k$ with $P_i$ full and $P_i \cap P_{i+1} \neq \emptyset$ so that $v_1 \in P_0$ and $v_2 \in P_k$. Hence there must be some full $P_i$ containing some vertex of $L$ and some vertex not in $L$, and this will verify that $L$ is not deformable.

Next, we show that if $\mu$ is not full, then $\mu$ is not rigid. Suppose first that there is some deficient polytope $P$ and some component $C_1$ of $\supp(\mu)$ such that $|P \cap C_1| \ge 2$. Pick additional components $C_i$, $i = 2, \dots, s$ intersecting $P$ until $\bar C:= \cup_{i=1}^s C_i$ has $|P \cap \bar C| > \mu(P)$. Note that $s \le \mu(P)$. Take $L = \ZZ(\bar C)$. Using Corollary \ref{unioncomp}, we get $\used(L) \ge |L| - s + \mu(P) \ge |L|$. So by Lemma \ref{lem:used} $\mu$ cannot be rigid.

Hence we may suppose that every deficient polytope $P$ meets every component of $\supp(\mu)$ in at most one lattice point. Pick any component $C$ of $\supp(\mu)$, and let $L = \{v: v \notin C\}$. This will be deformable: every full polytope will either be in $C$ and have none of its lattice points in $L$ or not be in $C$ and have all of its lattice points in $L$, while every deficient polytope $P$ will have at most one of its lattice points not in $L$ (it can have at most one lattice point in $C$ by assumption). Being deficient says that $P$ has $\mu(P) < |P| - 1$, so $|P \cap L| \ge |P| - 1 > \mu(P)$ as desired. We have shown that not full implies not rigid.

It remains to show that when $\mu$ is rigid (and hence also full), $\supp(\mu)$ is connected. But if $\supp(\mu)$ is disconnected, one sees that the lattice points of any component will form a suitable deformable $L$ .
\end{proof}

\begin{remark}
Both the lattice simplicial hypothesis and the hypothesis on $|\mu|$ are necessary in Theorem \ref{thm:con}, as shown in Figures \ref{fig:sing} and \ref{fig:count-hyp-nec}, respectively.
\end{remark}

We have observed that points in $\TT^n$ need not in general impose independent conditions on $\Delta$-hypersurfaces. However, a single point with multiplicity $|\Delta|-1$ always determines a unique hypersurface (see Figure \ref{fig:double-line} for an example). In fact, somewhat amusingly, putting points on top of each other always imposes independent conditions. More precisely:
\begin{proposition}
The set of coefficients of $\Delta$-hypersurfaces passing through a point $p$ with multiplicity $m$ is the support of a pure codimension $m$ polyhedral complex in $\TT^{|\Delta|}$.
\end{proposition}
\begin{proof}
The complex has a one-dimensional lineality space obtained by simultaneously scaling all the coefficients. Given any hypersurface meeting the point with multiplicity $m$, there are at least $m+1$ monomials which are minimal at $p$. Pick any $m+1$ of these, and now the coefficients of any other monomials can be increased.
\end{proof}

\begin{figure}
\begin{center}
\begin{tikzpicture}
\draw [<->] (-3.4,-2) -- (-.7,.7) -- (-.7, 2);
\draw [->] (-.7,.7) -- (.7, -.7) -- (2, -.7);
\draw [->] (.7,-.7) -- (-.6,-2);
\markpoint{0,0}
\markpoint{.7,-.7}
\markpoint{-.7,.7}
\markpoint{.2,-1.2}
\markpoint{-1.2,.2}
\node at (1,1) {$0$};
\node at (-1,-1) {$xy$};
\node at (-1.5,.7) {$x^2$};
\node at (1,-1.5) {$y^2$};
\path [fill=lightgray] (\ds, \dcgrid) -- (\ds + 2*\dcgrid, \dcgrid) -- (\ds + 2*\dcgrid, -\dcgrid) -- (\ds, \dcgrid);
\conicdualcomplex
\draw [very thick] (\ds + \dcgrid, 0) -- (\ds + 2*\dcgrid, \dcgrid);
\draw [very thick] (\ds, \dcgrid) -- (\ds + 2*\dcgrid, -\dcgrid);
\node at (\ds + \dcgrid, \dcgrid/1.7) {$1$};
\node at (\ds + 3*\dcgrid/1.8, 0) {$1$};
\end{tikzpicture}
\end{center}
\caption{A non-\smooth{} conic with a rigid weighting.}
\label{fig:sing}
\end{figure}
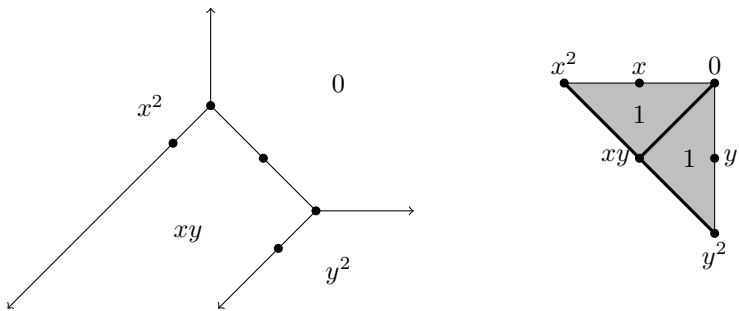

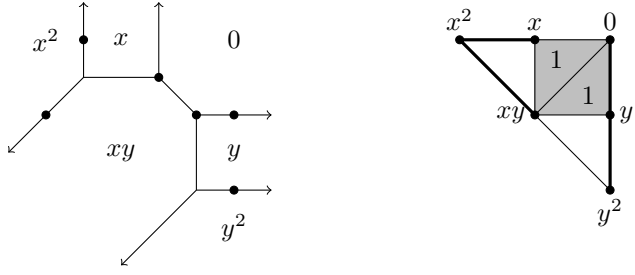
\begin{figure}
\begin{center}
\begin{tikzpicture}
\path [fill=lightgray] \czz -- \cxz -- \cxy -- \cyz -- \czz;
\otherconic
\markpoint{0,.5}
\markpoint{.5,0}
\markpoint{-1,1}
\markpoint{-1.5, 0}
\markpoint{1,0}
\markpoint{1,-1}
\draw [very thick] \cxz -- \cxx -- \cxy;
\draw [very thick] \czz -- \cyz -- \cyy;
\coym{1}
\coxm{1}
\end{tikzpicture}
\end{center}
\caption[This weighting is not full but is rigid]{This weighting is not full but is rigid, showing that the hypothesis $|\mu| = |\Delta|-1$ is necessary in Theorem \ref{thm:con}.}
\label{fig:count-hyp-nec}
\end{figure}

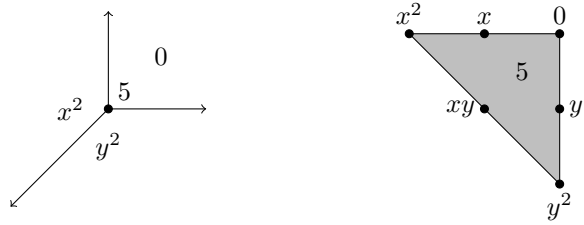
\begin{figure}
\begin{center}
\begin{tikzpicture}
\draw [<->] (0,1.3) -- (0,0) -- (1.3,0);
\draw [<-] (-1.3,-1.3) -- (0,0);
\node at (.7,.7) {$0$};
\node at (-.5, 0) {$x^2$};
\node at (0,-.5) {$y^2$};
\markpoint{0,0}
\node [above right] at (0,0) {$5$};
\path [fill=lightgray] (\ds, \dcgrid) -- (\ds + 2*\dcgrid, \dcgrid) -- (\ds + 2*\dcgrid, -\dcgrid) -- (\ds, \dcgrid);
\conicdualcomplex
\node at (\ds +3*\dcgrid/2,\dcgrid/2) {$5$};
\end{tikzpicture}
\end{center}
\caption[An extreme example of a weighting with higher multiplicity]{An extreme example of a weighting with higher multiplicity.}
\label{fig:double-line}
\end{figure}

\section{Tropical linear algebra with multiplicities} \label{sec:trop-la}

Our goal in this section to state and prove generalizations of Theorems \ref{str_square} and \ref{thm:str_n-1xn} that take into account our notion of multiplicity.

\subsection{Definitions and statements}

We begin by introducing the definition of a \emph{tropical-weighted} (henceforth \emph{tw}-) \emph{matrix}.
\begin{definition} \label{tmdef}
A $K \times N$ \matrix{} $A$ is a $M \times N$ matrix, together with a partition $\sum_{i=1}^{M} m_i = K$, $m_i \ge 1$. 

For a $K \times N$ \matrix{}, we say that a vector $\vx$ is in the \kernal{} of $A$ if for each row $i$
\[
    \min_j \{x_j + A_{ij}\}
\]
is achieved at least $m_i+1$ times.

The \perm{} of an $N \times N$ \matrix{} $A$ is 
\[
    \min_{\mathcal I} \left\{\sum_{i} \sum_{j \in I_i} A_{ij}\right\}, \label{eq:minpart}
\]
where the minimum is over all partitions $\mathcal I = \{I_i\}_{i=1}^{K}$ of $\{1, \dots, N\}$ with $|I_i|=m_i$. The $I_j$ themselves are ordered, but the elements of $I_j$ are not. We say $A$ (or its \perm{}) is \sing{} if the minimum above is obtained more than once.
\end{definition}

\begin{remark}
	One could also describe the \perm{} by saying that it is the minimum sum obtained by choosing, for each $i$, $m_i$ entries from row $i$ so that exactly one entry is chosen from each column. We also note that the value of the \perm{} is the same value that would be obtained by repeating the $i$th row $m_i$ times and taking the usual tropical permanent \eqref{tropperm} (but such a permanent would automatically be singular if $M \neq K$).
\end{remark}

\begin{example}
Consider the $3 \times 3$ \matrix{}
\[
\begin{array}{c} 2 \\ 1\end{array} \begin{bmatrix} 0&0&5 \\ 2 & 1 & 1 \end{bmatrix}.
\]
Here we use the numbers on the left to indicate the $m_i$. The \perm{} of this matrix is 1, obtained by the partition $\{\{1,2\},\{3\}\}$, that is, one chooses the two zeros in the first row and the rightmost 1 in the second row. One can easily see that this is the uniquely minimizing among partitions $\mathcal I$, so this \matrix{} is \nsing{}. 
\end{example}

\begin{definition}
When $A$ is $(N-1) \times N$, a maximal \minor{} of $A$ is the \perm{} of the $(N-1) \times (N-1)$ \matrix{} formed by deleting a column of $A$.
\end{definition}

Our main theorems of this section are:
\begin{theorem}
\label{square}
An $N \times N$ \matrix{} is \sing{} if and only if there is a vector in its \kernal{}.
\end{theorem}

\begin{theorem}[Tropical Weighted Cramer's Rule]
\label{n-1xn}
Let $A$  be an $(N-1) \times N$ \matrix{}. Then the vector of maximal minors is in the \kernal{} of $A$. Furthermore, this vector is unique up to scaling if and only if every maximal \minor{} of $A$ is \nsing{}.
\end{theorem}

\begin{example} Consider the $3 \times 4$ \matrix{}:
\[
 3 \begin{bmatrix}1 & 2 & 3 & 4\end{bmatrix}.
\]
Here the 3 on the left indicates that $m_1=3$. The vector of maximal minors is $\begin{bmatrix}9 & 8 & 7 & 6\end{bmatrix}$, which is in the \kernal{}, as the values of $x_j + A_{ij}$ are $\begin{bmatrix}10 & 10 & 10 & 10\end{bmatrix}$. One easily sees that this is unique (up to tropical scaling), and indeed, the minors are \nsing{}. (In fact, a $3 \times 3$ \matrix{} with $m_1 =3$ is always \nsing{}.)
\end{example}

\begin{example}
As a second example, consider
\[
\begin{array}{c} 2 \\ 1\end{array} \begin{bmatrix} 0&0&0&0 \\ 2 & 1 & 1 & 3\end{bmatrix}
\]
The vector of maximal minors is $\begin{bmatrix}1 & 1 & 1 & 1\end{bmatrix}$. However, the vector $\begin{bmatrix}1 & 1 & 1 & 99\end{bmatrix}$ is also in the \kernal{}. And indeed, the 4th minor is \sing{}, with $\{1,2\},\{3\}$ and $\{1,3\},\{2\}$ being partitions achieving the minimum value of 1 and having $\begin{bmatrix} 0& 0& 0 \end{bmatrix}$ in its \kernal{}.
\end{example}

Theorems \ref{square} and \ref{n-1xn} will be proved in Sections \ref{squareproof} and \ref{rectproof}. Notice that in the case that $m_i = 1$ for all $i$, these theorems recover Theorems \ref{str_square} and \ref{thm:str_n-1xn}. One might hope that the weighted theorems could be deduced easily from the unweighted version, however, as far as the author can tell, that does not seem to be the case. Hence we give a proof ``from scratch".

We can summarize the geometric consequences as follows.
\begin{corollary} \label{lahyp}
Let $\Delta$ be a $n$-dimensional lattice polytope and set $N = |\Delta|$. Let $\{\vx_i\}$ be a collection of $K$ points in $\TT^n$, and $m_i$ multiplicities so that $\sum m_i = N-1$. Let $A$ be the $(N-1) \times N$ \matrix{} formed by evaluating lattice points of $\Delta$ at $\vx_i$ and assigning weight $m_i$ to row $i$.  Then the tropical hypersurface with coefficients given by the maximal minors of $A$ has multiplicity at least $m_i$ at $\vx_i$ for all $i$. That hypersurface is unique if and only if the maximal \minor{}s of $A$ are \nsing{}. 
\end{corollary}

\subsection{Hypergraphs}
Graphs provide a convenient way to keep track of patterns formed when solving optimization problems using tropical matrices (see \cite{sturmfels_maximal_1993}). Because we are allowing multiplicities, we will need to work with hypergraphs.

\begin{definition}
A \emph{hypergraph} $G$ is a set $V$ of vertices, together with a set $E$ of subsets of $V$ called edges. We say an edge $e$ \emph{touches} a vertex $v$ if $v \in e$. We require an edge to touch at least two vertices. The \emph{valency} of a vertex is the number of edges touching it.

A \emph{path} in a hypergraph is a sequence $v_1, e_1, v_2, e_2, \dots, v_n$, with $v_i \in V$ and $e_i \in E$ such that $v_i$ and $v_{i+1}$ are both in $e_i$.  There is an equivalence relation on $V$ where $v \sim w$ if there is a path with $v=v_1$ and $w=v_n$. An equivalence class of $V$, together with all edges touching any vertex in the class, is called a \emph{connected component} of $G$. The hypergraph is \emph{connected} if there is only one connected component. 

We call a path a \emph{closed walk} if $v_1=v_n$. A closed walk is a \emph{simple cycle} if no vertex or edge is repeated besides $v_1=v_n$. Simple cycles are considered up to cyclic permutation, that is, the simple cycle $v_1, \dots, v_n=v_1$ is not considered distinct from $v_2, \dots, e_{n-1}, v_n=v_1, e_1, v_2$.

Notice that our definitions allow multi-edges, but not self-edges (``loops"). We also consider the length 2 cycles formed by multi-edges as legitimate cycles.

A hypergraph is called a \emph{graph} if every edge touches two vertices. 

The \emph{edge total} $e(G)$ of a hypergraph is $\sum_{e \in E} (|e|-1)$. Notice that if $G$ is a graph then the edge total is equal to the number of edges. The vertex total is $v(G) = |V|$.
\end{definition}

\begin{notation} \label{graphify}
The following construction allows us to use results about graphs to help prove results about hypergraphs. Given a hypergraph $G$, we may construct a graph $\hat G$ by replacing each edge $e \in E_G$ with a tree connecting all vertices touched by $e$. (Of course $\hat G$ is not uniquely determined by $G$.) $E_{\hat G}$ comes equipped with a surjection $\psi$ onto $E_G$, where an edge of $\hat G$ is associated to the edge of $G$ that gave rise to it. Notice that $e(G) = e(\hat G)$ and $\hat G$ is connected if and only if $G$ is.  
\end{notation}

\begin{proposition} \label{hasloop}
A hypergraph $G$ has a simple cycle if and only if 
\[e(G) +\mbox{\# of components } \ge v(G) + 1.\]
\end{proposition}
\begin{proof}
The result is standard for graphs. If $G$ has a simple cycle, then so does $\hat G$, so the inequality is satisfied.

Now, assume we have the inequality. Then $\hat G$ has a simple cycle $v_1, e_1, v_2, e_2, \dots, v_n$. This gives a closed walk $v_1, f_1, v_2, f_2, \dots, v_n$ in $G$, where $\psi(e_i) = f_i$. Since $\psi$ only contracts trees, it cannot be the case that the $f_i$ are all the same, so by changing the base point of the closed walk, we may assume that $f_1 \neq f_{n-1}$. If the $f_i$ are not all distinct, then let $f_s=f_t$ with $s<t$ and replace the closed walk with $v_1, \dots, v_s, f_s = f_t, v_t, \dots, f_{n-1},v_n$. Repeat this process if necessary until one obtains a simple cycle.
\end{proof}

\begin{definition}
For this paper, we define a \emph{good orientation} of a graph to be an orientation such that every vertex has exactly one outgoing edge. It is easy to see that this is equivalent to giving a bijection $out:V \rightarrow E$ such that $out(v)$ is an edge touching $v$. This is a special case of a good orientation of a hypergraph, which is a surjection $out:V \rightarrow E$ such that
\begin{itemize}
\item $out(v)$ touches $v$
\item $|out^{-1}(e)| = |e|-1$. %i think i need the -1 here!!
\end{itemize}
\end{definition}
Notice that a good orientation for $\hat G$ gives one for $G$ by composing with $\psi$.

\begin{proposition} \label{goodorientation}
If $G$ is a connected hypergraph with $e(G) = v(G)$, then it has at least two good orientations. 
\end{proposition}
\begin{proof}
It is a standard result for graphs that there is a unique simple cycle in $\hat G$. There are two distinct good orientations for this cycle---one for each way of going around the cycle. Excluding the edges in the cycle from $\hat G$ gives a forest, with each tree having a distinguished vertex that was part of the cycle. One can extend the good orientation to the forest by taking the flow into the distinguished vertex (each vertex $v$ will be the source of the first edge in the unique path from $v$ to the distinguished vertex). 

It remains to check that these orientations are distinct after composing with $\psi$.  But this could only fail if $\psi$ collapses the cycle---which cannot happen because $\psi$ only collapses trees.
\end{proof}

\begin{definition} \label{def:hypergraphtypes1}
	When considering a \matrix{} $A$ as in Definition \ref{tmdef},
	we say a hypergraph is a \emph{linkage hypergraph} if it has $N$ vertices (corresponding to columns of $A$) and $M$ edges (corresponding to rows of $A$) such that edge $i$ touches exactly $m_i + 1$ vertices.
	
	Now let $A$ be a non-negative $K \times N$ \matrix{}. A linkage hypergraph complementary to $A$ is a linkage hypergraph such that $A_{ij}=0$ whenever $i$ touches vertex $j$.
	
	Notice that a non-negative $A$ has a linkage hypergraph complementary to it if and only if $\mathbf 0$ is in its \kernal{}.
\end{definition}

\begin{corollary} \label{orientation-sing}
If $A$ is a non-negative $N \times N$ \matrix{} and has a connected complementary linkage hypergraph, then $A$ is \sing{}.
\end{corollary}
\begin{proof}
By Proposition \ref{goodorientation} the complementary linkage hypergraph has two good orientations. A good orientation of a complementary hypergraph is the same as a choice of partition $\mathcal I$ that certifies that $\permop A \le 0$ (via $I_i = out^{-1}(i)$). But as $A$ was assumed non-negative, we have $\permop A \ge 0$, so $\permop A = 0$ and $A$ is \sing{}.
\end{proof}

\subsection{Some lemmas and the proof of Theorem \ref{square}}
\label{squareproof}

\begin{lemma} \label{rescalesquare}
For any $N \times N$ \matrix{} $A$, one can rescale the rows and columns so that every entry is non-negative and $\permop A = 0$.
\end{lemma}
\begin{proof}
First we do the unweighted case when $m_i=1$ for all $i$. By permuting rows and columns and rescaling, we may assume that the diagonal is the minimizing permutaion and that every entry on the diagonal is 0. The problem now is to show that the matrix can be further rescaled to eliminate the negative entries without disturbing the diagonal. 

We construct a labeled graph $G(A)$ by taking complete di-graph on $N$ labeled vertices, with the edge from $i$ to $j$ labeled by $A_{ij}$. For any directed path $\gamma$ in $G(A)$, we define its path sum $p(\gamma)$ to be the sum of all the labels of the edges in the path. Notice that the path sum of any simple cycle corresponds to a (cyclic) permutation in the formula for the tropical permanent, so we see that the tropical permanent being equal to zero implies that the path sum of any simple cycle is non-negative.

Notice that if we subtract $c$ from row $i$ and add $c$ to column $i$, we still have a matrix with each diagonal entry equal to zero and tropical permanent equal to zero. %In the graph $G(A)$, this operation corresponds to adding $c$ to the label of each outgoing edge of vertex $i$ and subtracting it from each incoming edge of vertex $i$.  Hence, it is enough to show that we can pick $c_i$ for each vertex $i$ such that after this operation is performed on each vertex, every edge has non-negative label. 
Hence it is enough to pick $c_i$ so that $A_{ij} - c_i + c_j \ge 0$.

Pick 
\[
c_i = \min \left\{p(\gamma):\gamma \mbox{ is a simple path starting at }i\right\}.
\]
Here a simple path is one that uses any vertex at most once. 

Now we check $A_{ij} -c_i + c_j \ge 0$. There is a simple path $\gamma$ from vertex $j$ with $p(\gamma)=c_j$. If $\gamma$ does not meet vertex $i$, then we can form a simple path $\gamma'$ starting at $i$ by concatenating the edge from $i$ to $j$ with $\gamma$. Then $c_i \le  p(\gamma') = c_j + A_{ij}$, as desired. If $\gamma$ meets vertex $i$, we can split $\gamma$ at vertex $i$ into two simple paths: $\gamma_1$ from $j$ to $i$ and $\gamma_2$ the rest of it (possibly trivial). We have $p(\gamma_1) + p(\gamma_2) = c_j$. Now, we can form a simple cycle by concatenating $\gamma_1$ with the edge from $i$ to $j$. Hence $p(\gamma_1) + A_{ij} \ge 0$.   Furthermore, $\gamma_2$ is a simple path from $i$, so $c_i \le  p(\gamma_2)$. Hence
\[
 A_{ij} - c_i + c_j \ge -p(\gamma_1) - c_i + p(\gamma_1) + p(\gamma_2) \ge 0,
\]
as desired.

For the weighted case, we can permute and rescale it so that the optimizing partition is the ``diagonal" one, i.e., $I_1=\{1, \dots, m_1\}$, $I_2=\{m_1+1, \dots m_1 + m_2\}$, etc., and so that each of these entries are 0. Then we form a square (nonweighted) tropical matrix by repeating the $i$th row $m_i$ times. Then run the argument above. Notice that two vertices (say $i$ and $k$) corresponding to a row repeated in this way will have all edges between them labeled 0. Hence $c_i = c_k$, and the resulting matrix will have the same pattern of repeated rows. One can then reidentify these rows to get the desired \matrix{}.
\end{proof}

The following lemma is used in the proof of both theorems. The full strength of it is needed only for the $(N-1) \times N$ case.
\begin{figure}
\begin{center}
\begin{tikzpicture}[-,>=stealth',shorten >=1pt,auto,node distance=3cm,
                    thick,main node/.style={circle,draw,font=\sffamily\Large\bfseries}]

  \node[main node] (2) {2};
  \node[main node] (1) [right of=2] {1};
  \node[main node] (5) [right of=1] {5};
  \node[main node] (3) [below right of=2] {3};
  \node[main node] (4) [right of=3] {4};

  \path[every node/.style={font=\sffamily\small}]
    (1) edge node {1} (2)
        edge node {2} (3)
    (2) edge node {3} (3)
    (3) edge node {4} (4);
\end{tikzpicture}
\raisebox{9ex}{
$
\begin{bmatrix}
0 & 0 & 3 & 3 & 3 \\
0 & 4 & 0 & 4 & 5 \\
5 & 0 & 0 & 3 & 4 \\
6 & 7 & 0 & 0 & 1
\end{bmatrix}
$}
\vspace{3ex}

\begin{tikzpicture}[-,>=stealth',shorten >=1pt,auto,node distance=3cm,
                    thick,main node/.style={circle,draw,font=\sffamily\Large\bfseries}]

  \node[main node] (2) {2};
  \node[main node] (1) [right of=2] {1};
  \node[main node] (5) [right of=1] {5};
  \node[main node] (3) [below right of=2] {3};
  \node[main node] (4) [right of=3] {4};

  \path[every node/.style={font=\sffamily\small}]
    (1) edge node {1} (2)
        edge node {2} (3)
    (2) edge node {3} (3)
    (5) edge node {4} (4);
\end{tikzpicture}
\raisebox{9ex}{
$
\begin{bmatrix}
0 & 0 & 3 & 3 & 2 \\
0 & 4 & 0 & 4 & 4 \\
5 & 0 & 0 & 3 & 3 \\
6 & 7 & 0 & 0 & 0
\end{bmatrix}
$
}
\vspace{3ex}

\begin{tikzpicture}[-,>=stealth',shorten >=1pt,auto,node distance=3cm,
                    thick,main node/.style={circle,draw,font=\sffamily\Large\bfseries}]

  \node[main node] (2) {2};
  \node[main node] (1) [right of=2] {1};
  \node[main node] (5) [right of=1] {5};
  \node[main node] (3) [below right of=2] {3};
  \node[main node] (4) [right of=3] {4};

  \path[every node/.style={font=\sffamily\small}]
    (1) edge node {2} (3)
        edge node {1} (5)
    (2) edge node {3} (3)
    (5) edge node {4} (4);
\end{tikzpicture}
\raisebox{9ex}{
$
\begin{bmatrix}
0 & 0 & 3 & 1 & 0 \\
0 & 4 & 0 & 2 & 2 \\
5 & 0 & 0 & 1 & 1 \\
8 & 9 & 2 & 0 & 0
\end{bmatrix}
$
}
\end{center}
\caption{An example of running the proof of Lemma \ref{game} on an explicit (unweighted) matrix. Node labels refer to columns, edge labels refer to rows. To move from the first picture to the second, we used case (2), and from the second to the third used case (1). In the last picture, one could just as well have edge 1 connect nodes 2 and 5.}
\label{fig:gameproof}
\end{figure}
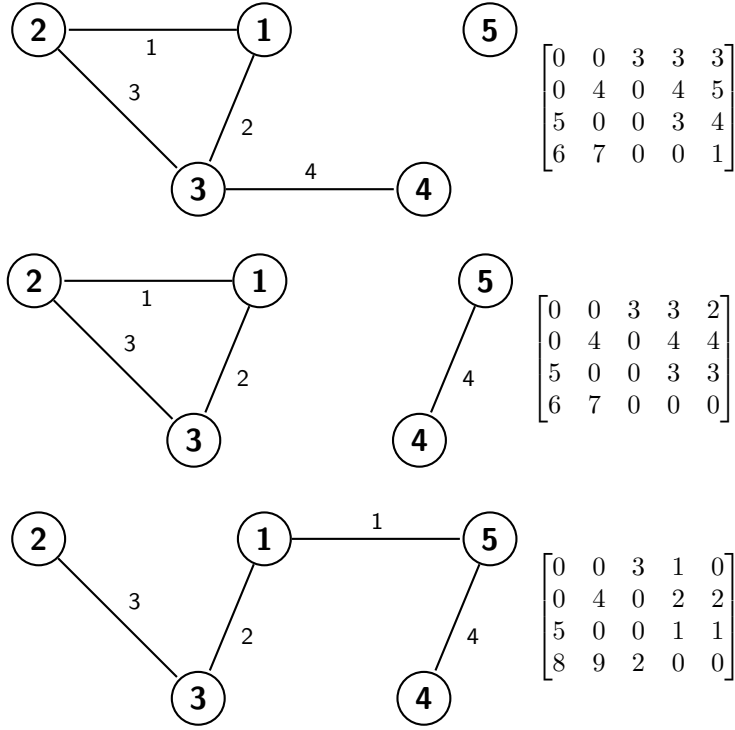
\begin{lemma} \label{game}
Suppose $A$ is a non-negative $K \times N$ \matrix{} with $K \ge N-1$.  Let $G$ be a linkage hypergraph complementary to $A$. If $G$ is not connected, then we may rescale rows and columns to produce $A'$ that is non-negative and has two complementary connected linkage hypergraphs.
\end{lemma}
\begin{proof}
See Figure \ref{fig:gameproof} for an illustration of the proof.

Given any connected component of $G$ with edges $E$ and vertices $V$, one can perform the following operation (*): Subtract from the rows of $A$ corresponding to $E$ and add to the columns of $A$ corresponding to $V$ to form $A'$. Do this as much as possible without violating non-negativity, i.e., add and subtract $\min_{i \in E, j \notin V} A_{ij}$. Notice that $G$ is also complementary to $A'$. Let $(i^*,j^*)$ be the indices achieving the minimum above. Then $A'_{i^*j^*}=0$, so we may replace any vertex touching edge $i^*$ with vertex $j^*$ and get a new linkage hypergraph $G'$ that is complementary to $A'$.

Let $M(G)$ be the minimal number of vertices of a connected component of $G$ that has a cycle (such a component exists by Proposition \ref{hasloop}), and let $C$ be one of these components. We will perform operation (*) with this component. To choose the vertex touching $i^*$ to be replaced by $j^*$, we must consider the components $\{C_\alpha\}$, $\alpha = 1, \dots, k$, formed from $C$ by deleting edge $i^*$. 
\begin{enumerate}
	\item \label{item:connect} If $k \le m_{i^*}$, then there is some $C_\beta$ that contains at least two vertices of edge $i^*$. Pick one of these vertices and replace it with $j^*$ to form $G'$. Then we see that $G'$ has fewer components than $G$.
	\item If $k = m_{i^*}+1$, then we cannot reduce the number of components of $G$. Instead, let $\bar C$ be the subgraph $\bigcup_\alpha C_\alpha$ and notice that  
	\begin{align*}
    	e(\bar C) + \mbox{ \# of components of } \bar C &= e(\bar C) + m_{i^*} + 1 \\
	    &= e(C) + 1 \\
	    &\ge v(C) + 1 \\
	    &=v(\bar C) + 1
	\end{align*}
	so Proposition \ref{hasloop} implies that there is some $C_\beta$ that has a cycle. Pick the vertex of $C_\beta$ touched by $i^*$ and replace it with $j^*$. Then we see that $C_\beta$ will be a component of $G'$, which will imply that $M(G') < M(G)$. 
\end{enumerate}
Hence at each step, either the number of connected components decreases, or $M(G)$ decreases. Since $M(G)$ is non-negative, it cannot decrease forever, so eventually the number of components must decrease. Furthermore, notice that the final step will be of type \eqref{item:connect}, and so there are two distinct choices giving hypergraphs complementary to $A'$.
\end{proof}

\begin{notation} \label{not:Ax}
If a matrix $A$ has a vector $\vx$ in its \kernal{}, one can scale column $j$ by $x_j$, forming a new matrix which has $\mathbf 0$ in the \kernal{}. Then, one can scale each row so that the minimal entry is zero. We call the matrix obtained this way $A_\vx$. 
\end{notation}

\begin{proof}[Proof of Theorem \ref{square}]
First, assume that $A$ has a vector $\vx$ is its \kernal{}. Then, there is a linkage hypergraph complementary to $A_\vx$. We may assume that it is connected (after possibly rescaling) by Lemma \ref{game}. Hence by Corollary \ref{orientation-sing} $A$ is \sing{}.

Now, suppose we are given two partitions $I$ and $J$ . By Lemma \ref{rescalesquare} we may assume that $A$ has non-negative entries that every entry corresponding to either partition is 0, that is $A_{ij} = 0$ whenever $j \in I_i$ or $j \in J_i$. Now, take $K = \{k:J_k \neq I_k\}$, and let $L = \{\ell : \ell \in J_k \cup I_k \mbox{ for some } k \in K\}$.

One can check that $\sum_{k \in K} m_k = |L|$ (that is, the submatrix determined by $L$ and $K$ is ``tw-square''). 

Now notice that $\mathbf 0_K$ is in the tw-kernel of the submatrix determined by $K$ and $L$. We wish to extend this submatrix. Take  $\epsilon = \min_{i \notin K,j \in L} A_{ij}$ and then add $\epsilon$ to each row in $K$ and subtract $\epsilon$ from each column in $L$. We see now that there is a zero entry in some row $i^* \notin K$ and column $j^* \in L$. We can append $i^*$ to $K$ and $J_{k^*} = I_{k^*}$ to $L$ and we still have the property that $\mathbf 0_K$ is in the tw-kernel of the submatrix determined by $K$ and $L$. We may repeat this until $|L| = |K| =  N$. The resulting matrix will have $\mathbf 0$ in its tw-kernel, so the theorem is proved.
\end{proof}

\subsection{Stochastic and transportation polytopes and the proof of Theorem \ref{n-1xn}}
\label{rectproof}
\begin{definition} \label{def:transpoly}
Let $D$ be the polytope of weighted doubly stochastic square $(N-1) \times (N-1)$ \matrices{}---that is, we require $\sum_j A_{ij} = 1$ for each $i$, and $\sum_i m_iA_{ij} = 1$ for each $j$.
\end{definition}

\begin{definition}	
Let $Y$ be a non-negative $K \times N$ \matrix{}. The support hypergraph of $Y$ is the hypergraph formed by having edge $i$ touch vertex $j$ whenever $Y_{ij} > 0$. If this prescription would cause an edge to touch one or zero vertices, omit that edge.
\end{definition}

\begin{lemma} \label{lem:notvert}
The support hypergraph of $Y$ contains a simple cycle if and only if $Y$ is not uniquely determined by the data of its support, row sums, and weighted column sums.
\end{lemma}
\begin{proof}
Let $v_1, e_1, v_2, e_2, v_3, \dots, e_\ell, v_{\ell+1}=v_1$ be a simple cycle of the support hypergraph of $Y$. For each $k=1, \dots, \ell$, one can add a small amount $\epsilon/m_{e_k}$ to entries $(e_k, v_k)$ and subtract $\epsilon/m_{e_{k+1}}$ from entries $(e_{k+1}, v_k)$. This new matrix still has the same row and weighted column sums and the same support.
	
Now suppose $Y'\neq Y$ has the same support and row and weighted column sums as $Y$. By taking convex combinations of $Y$ and $Y'$, we see that there is a deformation of the non-zero entries of $Y$ that preserves the row and column sums. Suppose $Y_{e_1,v_1}$ decreases in this deformation. Thus, there must be a $v_2$ with $Y_{e_1,v_2}$ increasing (in order to preserve the $e_1$ row sum). Then, there must be an $e_2$ with $Y_{e_2,v_2}$ decreasing (to preserve the $v_2$ column sum). Continue in this way until some vertex or edge is repeated. If a vertex is repeated, say $v_m = v_s$ with $m<s$, then $v_m, e_m, v_{m+1}, \dots, e_{s-1}, v_s$ is a simple cycle. If an edge is repeated, say $e_m = e_s$, then $v_m, e_m, \dots, v_s, e_s, v_m$ is a simple cycle.
\end{proof}

%%%%%
\begin{remark} \label{rem:vert}
Given a polytope defined by a collection of equality and inequality constraints, we see that a point of the polytope is a vertex if it is the unique solution to the system of equations formed by replacing all the tight inequality constraints with equalities (and keeping all the equality constraints). Since in our situation, the inequality constraints are $A_{ij} \ge 0$, we see that a vertex of $D$ is a matrix whose entries are uniquely determined by its support and the constraints on the row and column sums. By Lemma \ref{lem:notvert}, we see that this is equivalent to the support hypergraph not containing a cycle.
\end{remark}
%%%%%

\begin{definition}
For any two $K \times N$ \matrices{}, define  the weighted inner product $\left<\cdot,\cdot\right>_w$ by
\[
	\left<A,Y\right>_w = \sum_{i,j} m_i A_{ij}Y_{ij}.
\]
%Notice that (tropically) scaling a row $k$ (or column $\ell$) of $A$ by a constant $\alpha$ changes the value of the inner product by $\sum$ (or $\alpha K$), respectively.
\end{definition}

\begin{lemma} \label{square-perm}
We have:
\begin{enumerate}
\item The vertices of $D$ are in bijection with partitions $\mathcal I$ as in \eqref{tmdef}. Explicitly, a vertex corresponding to $\mathcal I$ is a matrix $Y_{\mathcal I}$  with $(Y_{\mathcal I})_{ij} = \frac{1}{m_i}$ if $j \in I_i$ and $0$ otherwise. 
\item
\[
 \min_{Y \in D} \left< A, Y \right>_w = \permop(A)
\] 
and the minimizing $Y$ is unique if and only if $A$ is \nsing{}.
\end{enumerate}
\end{lemma}
\begin{proof}
To prove the second statement from the first, note that 
\[
 \left< A, Y_{\mathcal I} \right>_w = \sum_{i} \sum_{j \in I_i} A_{ij}.
\]
Thus we see that the \perm{} is equal to $\min_{\mathcal I} \left<A,Y_{\mathcal I}\right>_w$. Since the minimum over $D$ is achieved at a vertex, we see that (2) follows from (1).

To prove (1), we first we claim that $Y_{\mathcal I}$ is a vertex of $D$. Since each column has only one non-zero entry, that entry is determined the column sum, hence $Y_{\mathcal I}$ is determined by its support and is a vertex by Remark \ref{rem:vert}.

Now start with any $Y$ a vertex of $D$. Suppose there is pair $k,\ell$ such that $Y_{k\ell}>0$ and column $\ell$ has more than one nonzero entry, or equivalently $Y_{k\ell} < \frac{1}{m_k}$. All entries in row $k$ are less than or equal to $\frac{1}{m_k}$, and they all sum to 1. It follows that row $k$ must have at least two entries strictly less than $\frac{1}{m_k}$. Both columns containing these entries have more that one nonzero entry each. Translating this into the language of support hypergraphs, we have shown that if an edge contains a (hypergraph) vertex with valence at least 2, then it contains another (hypergraph) vertex with valence at least 2. It follows then that the support hypergraph contains a simple cycle, contradicting Lemma \ref{lem:notvert} (see Remark \ref{rem:vert}).  We conclude then that any vertex of $D$ has exactly one nonzero entry in each column. From this it quickly follows that $Y$ has the form of $Y_{\mathcal I}$ .
\end{proof}

\begin{definition}
The weighted transportation polytope $T$ is the set of non-negative $(N-1) \times N$ \matrices{} with row sums equal to $N$ and weighted column sums equal to $N-1$, that is $\sum_j A_{ij} = N$ for each $i$, and $\sum_i m_iA_{ij} = N-1$ for each $j$.

There are $N$ embeddings $\phi_1, \dots, \phi_N$ of the space of $(N-1)\times(N-1)$ \matrices{} into the space of $(N-1) \times N$ \matrices{} given by inserting a column of zeros. Notice that the Minkowski sum $\sum_{j=1}^N \phi_j(D)$ is contained in $T$.

Notice that Remark \ref{rem:vert} applies also to $T$.
\end{definition}

\begin{remark} \label{rmk:trans-poly}
The (unweighted) transportation polytope can be given the following ``physical" interpretation. Suppose one has $N$ factories, each of which produce $N-1$ units of a product. Suppose there are $N-1$ cities, each of which consume $N$ units of the product. The cost to transport a unit of the product from a factory to a city is given in the matrix $A$. Then $T$ is the feasible region for this optimization problem and $Y \mapsto \left<A,Y\right>$ is the cost function.
\end{remark}

\begin{lemma} \label{lem:vertofT}
There is a bijection between vertices of the transportation polytope $T$ and connected linkage hypergraphs $G$, given by taking the support hypergraph of a vertex. Every vertex of $T$ is in the Minkowski sum $\sum_{j=1}^N \phi_j(D)$.
\end{lemma}
\begin{proof}
First we check that the support hypergraph of a vertex $Y$ is a connected linkage hypergraph. We claim that no row $i$ can have less than $m_i +1$ nonzero entries. Otherwise, since the row sums are $N$, some entry must be greater than or equal to $N/m_i$. But then the weighted column sum of the column containing that entry is at least $N$, so $Y$ could not be in $T$. 

Now, if any row has too many nonzero entries, then the support hypergraph must by Proposition \ref{hasloop} contain a simple cycle and $Y$ would not be a vertex (see again Remark \ref{rem:vert}). So we conclude that the support hypergraph of $Y$ is a linkage hypergraph.

Next, we start with a connected support linkage hypergraph $G$ and construct a vertex $Y$ of $T$ with support hypergraph $G$. Form $\hat G$ as in Remark \ref{graphify}. Notice that $\hat G$ is a tree, so for each vertex $j$ we can consider the ``flow" function $\hat f_j: V(\hat G) -\{j\} \rightarrow E(\hat G)$ that assigns to each vertex $\ell \neq j$ the first edge in the unique path from $\ell$ to $j$. Let $E_{xy}$ be the matrix with zeros everywhere except for a $1$ in the $(x,y)$ entry. Let $f_j(\ell) = \psi \circ \hat f_j(\ell)$ (here $\psi$ is as in Remark \ref{graphify}) and define
\[
    D_j = \sum_{\ell \neq j} \frac{1}{m_{f_j(\ell)}} E_{f_j(\ell), \ell}.
\]
It is easy to check $D_j \in \phi_j(D)$, so $Y:=\sum_j D_j$ is contained in the Minkowski sum and thus in $T$. We also see that $Y$ has support hypergraph $G$. By Proposition \ref{hasloop}, $G$ has no cycle, so by Remark \ref{rem:vert} it is a vertex. This construction also shows how each vertex is in the Minkowski sum as claimed.
\end{proof}
%Now, suppose you have a vertex $Y$ of $T$. We must check that the support hypergraph is a connected linkage hypergraph. First we claim that no row $i$ can have less than $m_i +1$ nonzero entries. Otherwise, since the row sums are $N$, some entry must be greater than or equal to $N/m_i$. But then the weighted column sum of the column containing that entry is at least $N$, so $Y$ could not be in $T$. 

%Now, if any row has too many nonzero entries, then the support hypergraph must by Proposition \ref{hasloop} contain a simple cycle and $Y$ would not be a vertex (see again Remark \ref{rem:vert}). So we conclude that the support hypergraph of $Y$ is a linkage hypergraph.

%If the support hypergraph of $Y$ is not connected, then again by Proposition \ref{hasloop} it contains a simple cycle. We conclude that the support hypergraph of $Y$ is a connected linkage hypergraph.

\begin{corollary} \label{lem:sumDi}
$T$ is precisely the Minkowski sum $\sum_{j=1}^N \phi_j(D)$.
\end{corollary}
\begin{proof}
We have noted that the Minkowski sum is contained in $T$. For the other inclusion, since both sets are convex polytopes, it is sufficient to show that any vertex of $T$ is in the Minkowski sum, which is the second statement in Lemma \ref{lem:vertofT}.
\end{proof}

\begin{lemma} \label{lem:opt}
Consider the optimization problem
\begin{equation} \label{eq:opt}
	\min_{Y\in T} \left<A,Y\right>_w.
\end{equation}
The solution to \eqref{eq:opt} is unique if and only if $A$ has \nsing{} minors.
\end{lemma}
\begin{proof}
By Corollary \ref{lem:sumDi}, we see that \eqref{eq:opt} can be rewritten as 
\begin{align}
	&\min_{\{Y_j \in \phi_j(D)\}_j} \left<A, \sum_j Y_j\right>_w \\
= \sum_j &\min_{Y_j \in \phi_j(D)} \left<A, Y_j\right>_w \label{eq:rewrite}
\end{align}
But by Lemma \ref{square-perm}, we know that the summands in \eqref{eq:rewrite} have solutions equal to the \minor{}s of $A$ which are unique if and only if the \minor{}s are \nsing{}.
\end{proof}

\begin{proof}[Proof of Theorem \ref{n-1xn}]
First we check that the vector of maximal minors is in the \kernal{}. By rescaling the columns, we may assume the the first row has every entry equal to 0. The \emph{lower} minors of $A$ are the minors obtained by deleting the first row and $m_1+1$ columns. 

Notice that the value of the $i$th maximal minor is equal to the smallest value of a lower minor contained in it.

Let $L$ be the set of $N-1-m_1$ indexes defining the minimal lower minor. Hence for any $i \notin L$, the $i$th maximal minor has this value, and is minimal among maximal minors. Hence we see that the vector of maximal minors is in the kernel of the first row. As the Theorem is invariant under permutation of rows, we are done.

Next, suppose $A$ has a singular minor. By Theorem \ref{square}, the singular minor has a vector in its \kernal{}. Extending this vector by inserting any sufficiently large entry will create a vector in the \kernal{} of $A$.

Now suppose $A$ has two elements $\vx$ and $\vy$ of its kernel. There is a linkage hypergraph complementary to $A_\vx$ (see Notation \ref{not:Ax}). If it is not connected, then by Lemma \ref{game} we may rescale to produce a non-negative $A'$ with two complementary connected linkage hypergraphs. Each of these hypergraphs corresponds to a vertex of $T$, and since they are complementary to $A'$ they achieve the optimal value of $0$ in \eqref{eq:opt}, so by Lemma \ref{lem:opt}, $A'$ and hence also $A$ has a \sing{} minor. The same argument applies to $A_\vy$.

%By Lemma \ref{game}, we may assume that the zero hypergraphs of $A_\vx$ and $A_\vy$ are connected. If either of these hypergraphs is not a tree, then by Lemma \ref{notvert} there is a solution to \eqref{eq:opt} that is not a vertex and hence not unique. 

So now we may assume that both $A_\vx$ and $A_\vy$ have connected complementary linkage hypergraphs. Again, each of these trees corresponds to an optimal solution to \eqref{eq:opt} by Lemma \ref{lem:vertofT}, so if they are distinct, we must have a singular minor.

So finally, we may assume that $A_\vx$ and $A_\vy$ both have the same connected complementary linkage hypergraphs. Let $\min^x_i = \min_j\{A_{ij} + x_j\}$ and $\min^y_i = \min_j\{A_{ij} + y_j\}$. By construction, we see that for any edge $i$ touching vertex $j$ in this hypergraph, we have $x_j +A_{ij}= \min^x_i$ and $y_j +A_{ij}= \min^y_i$, hence $x_j-y_j = \min_i^x - \min_i^y$. It follows that if there is an edge from $j$ to $k$ in the support hypergraph, then $x_k-y_k = x_j -y_j$. But since the hypergraph is connected, this implies that $\vx$ is a tropical scalar multiple of $\vy$.
\end{proof}

\begin{remark}
	\label{rmk:gap}
	In \cite{richter-gebert_first_2005}, something like Lemma \ref{game} appears to be being used implicitly. On page 19 (of the arXiv version), the authors start with a $(N-1) \times N$ tropical matrix $C$ that is assumed to have nonsingular minors. From Theorem 2.4 in \cite{sturmfels_maximal_1993}, one knows that there is an associated linkage tree (from which the optimizing permutations of the minors can be extracted). The matrix $C$ is then rescaled so that it is non-negative and has at least two zeros on each row and one zero in each column.  It is then claimed that the linkage tree of $C$ is complementary to (the rescaled version of) $C$. For smaller values of $N$ it is easy to see that if the zero patterns do not form a tree then there is a singular minor (with tropical permanent equal to 0), but for larger $N$ this is not obvious to us. For example, the pattern of zeros could be represented by a graph as in Figure \ref{fig:gap}.
	
	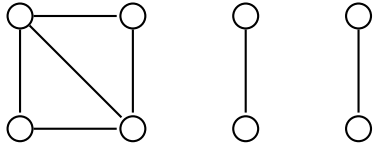
\begin{figure}[!htbp]
\begin{center}
\begin{tikzpicture}[-,>=stealth',shorten >=1pt,auto,node distance=1.5cm,
                    thick,main node/.style={circle,draw,font=\sffamily\Large\bfseries}]

  \node[main node] (1) {};
  \node[main node] (2) [right of=1] {};
  \node[main node] (3) [right of=2] {};
  \node[main node] (4) [right of=3] {};
  \node[main node] (5) [below of=1] {};
  \node[main node] (6) [right of=5] {};
  \node[main node] (7) [right of=6] {};
  \node[main node] (8) [right of=7] {};

  \path[every node/.style={font=\sffamily\small}]
    (1) edge node  {} (2)
        edge node  {} (6)
        edge node  {} (5)
    (5) edge node  {} (6)
    (2) edge node  {} (6)
    (3) edge node {} (7)
    (4) edge node {} (8);
\end{tikzpicture}
\end{center}
\caption{An interesting zero pattern}
\label{fig:gap}
\end{figure}

	In this case, there is no minor that is obviously singular. No minor has permanent equal to zero and which minor is singular will depend on the matrix itself and not only on the zero pattern. It seems to us that an argument like Lemma \ref{game} is necessary here.
\end{remark}

\bibliographystyle{plain}
\bibliography{Zotero}

\end{document}